\documentclass[11pt,a4paper]{amsart}
\usepackage[margin=29mm,headheight=14pt]{geometry}

\usepackage[T1]{fontenc}
\usepackage{lmodern}
\usepackage{amsmath,amssymb,amsthm,mathtools}
\usepackage{microtype}
\usepackage{enumitem}
\usepackage{booktabs,array}
\usepackage[dvipsnames]{xcolor}

\usepackage{mycom,myarrows}

\makeatletter
\def\l@subsection{\@tocline{2}{0pt}{2pc}{5pc}{}}
\makeatother

\usepackage[
  style=alphabetic, sorting=nyt, backend=biber, doi=true, isbn=false
]{biblatex}

\usepackage[unicode, colorlinks, linktocpage=true]{hyperref}
\hypersetup{
	linkcolor = blue,
	citecolor = Red,
	urlcolor = PineGreen,
	bookmarksnumbered = true
}

\numberwithin{equation}{section}
\newtheorem{theorem}{Theorem}[section]
\newtheorem{proposition}[theorem]{Proposition}
\newtheorem{lemma}[theorem]{Lemma}

\newtheorem{maintheorem}{Theorem}

\theoremstyle{definition}
\newtheorem{definition}[theorem]{Definition}

\theoremstyle{remark}
\newtheorem{remark}[theorem]{Remark}

\renewcommand{\Norm}{\mathrm N}
\setlist[enumerate]{label=\textup{(\roman*)},leftmargin=*,itemsep=3pt}
\allowdisplaybreaks[1]
\title[Nonvanishing of special $L$-values of cusp forms on $\GL(2)$]{Nonvanishing of special $L$-values of cusp forms on $\GL(2)$ with totally split prime power twists: without the class number hypothesis}
\usepackage{orcidlink}
\author{Ruichen Xu \,\orcidlink{0009-0003-4555-2116}}
\address[Ruichen Xu]{Faculty of Mathematics, University of Vienna. Oskar-Morgenstern-Platz 1, 1090 Vienna, Austria.}
\email{ruichen.xu@univie.ac.at}
\date{\today}
\begin{document}

\begin{abstract}
Let $F$ be a number field and $f$ a normalized cuspidal Hecke eigenform on $\GL(2)$ over $F$, with $k/2$ as the central critical point of its $L$-function $L(f,s)$. Let $\frp$ be a totally split prime ideal of $F$ lying above an odd prime $p$ and coprime to the level ideal of $f$. In this article, we show that $L(f\otimes\phi, k/2)$ does not vanish for sufficiently ramified Hecke characters of $F$ of $p$-power order and $\frp$-power conductor. Our result improves previous works of Jaesung Kwon and Hae-Sang Sun, in which the additional hypothesis that $p$ does not divide the class number of $F$ is required.
\end{abstract}

\maketitle

\tableofcontents
\section{Introduction}
\label{sec:intro}

Let $F$ be a number field and let $p$ be an odd prime number. The main object of this article is a normalized cuspidal Hecke eigenform $f$ on $\GL(2)$ over $F$ of weight $(\bfk, \bfm)$, archimedean type $J$, level $\frN$ and central character $\chi$ (see Definition \ref{def:cusp-form} for precise definitions, following \cite{Hida}). Let the point $s = \frac{k}{2} \in \frac{1}{2}\ZZ_{>0}$ be the central critical point for $f$ under the arithmetic normalization, determined by the weight $(\bfk, \bfm)$ of $f$. 

Let $\frp$ be a prime ideal of $F$ lying above $p$, and let $\phi$ be a Hecke character over $F$ of $p$-power order with $\frp$-power conductor. We denote the set of all such Hecke characters by $\Xi_{\frp}$. Let $L(f \otimes \phi, s)$ be the $L$-function attached to $f$ and $\phi$. The main result of this article is the following eventual nonvanishing property of these $L$-functions.

\begin{maintheorem}[{Theorem \ref{thm:nonvanishing}: Eventual nonvanishing theorem}] \label{main:A}
Assume that $p$ is an odd prime number such that $\frp \mid p$ is a totally split prime ideal \footnote{Here, by saying that $\frp \mid p$ is a totally split prime ideal of $F$, we mean both the ramification degree $e(\frp\mid p)$ and the inertia degree $f(\frp\mid p)$ are one. Let $\frd_{F}$ be the different ideal of $F$. This automatically implies that $\frp \nmid \frd_{F}$. The assumption $\frp \nmid \frd_{F}$ in \cite{KS} is therefore vacuous.} of $F$ such that $\frp$ is coprime to $\frN$. Then 
\[
 L\!\left(\frac{k}{2},f\otimes\phi\right)\ne0
\]
for every $\phi \in \Xi_{\frp}$ of sufficiently large conductor.
\end{maintheorem}

The method for proving such eventual nonvanishing theorems dates back to \cite{Rohrlich84, Rohrlich84-2}, and was adapted to our setting in \cite{KS,KwonCorrected}: one first considers the Galois-averaged $L$-function (see Definition \ref{def:averaged-L-function}) for $n > 2c$, where $c$ is the constant defined as $\QQ_{f}(\mu_{p}) \cap \QQ(\mu_{p^{\infty}}) = \QQ(\mu_{p^{c}})$:
\[
L_{\mathrm{av}}(f, \phi) := \frac{1}{\abs{G_\phi}q} \sum_{\substack{a\bmod q\\a\equiv c_\phi \bmod{p^c}}} \rme_q(a) \sum_{\sigma\in G_\phi} \overline{\phi}^{\,\sigma}(a)\calG_{0}(\phi^\sigma) L\!\left(\frac{k}{2},f\otimes\phi^\sigma\right)
\]
for $\phi \in \Xi_{\frp}$ of conductor $\frp^{n}$, where $G_{\phi} := \Gal(\QQ(\phi)/(\QQ_{f}(\mu_p)\cap \QQ(\mu_{p^{\infty}})))$, with $\QQ(\phi)/\QQ$ the number field containing all values of $\phi$, and $\QQ_{f}$ is the Hecke field of $f$ over $F$. Here $q = p^n$ and $\rme_{q}(a) := \exp(2 \pi \rmi a/q)$, while $\calG_{0}(\phi)$ is the Gauss sum and $c_{\phi}$ is a residue class modulo $p^{c}$ to be specified in the text. Hida's algebraicity theorem \cite[Theorem~8.1]{Hida} states that 
\[
 L(k/2,f\otimes\phi)=0
 \quad\Longleftrightarrow\quad
 L(k/2,f^\sigma\otimes\phi^\sigma)=0
\]
for any $\sigma \in \Aut_{\QQ}(\CC)$. Consequently, the eventual nonvanishing theorem, i.e. Theorem \ref{main:A}, reduces to the following estimate:
\[
\sup_{\phi\in\Xi_{\frp}, \, \cond(\phi)=\frp^n}
\abs{L_{\mathrm{av}}(f,\phi)-1}
\longrightarrow 0.
\]
This follows from the following main estimate.

\begin{maintheorem}[{Theorem \ref{thm:average}: The main estimate}] \label{main:B}
Assume the hypotheses of Theorem \ref{main:A}. 
Let $\phi\in\Xi_{\frp}$ have conductor $\frp^n$, where $n>2c$. For $\varepsilon>0$ with $t=\frac{7}{64}+\varepsilon<1/4$, put $s_0=1/2-t$ and let $q=p^n$. For $q<y<q^2$, we have the following estimate:
\begin{equation*}\label{eq:main-average-intro}
\begin{aligned}
L_{\mathrm{av}}(f,\phi)
&=\frac{V_{1,k/2}(1/y)}{\Gamma(k/2,f)}\\
&\quad+O_{F,f,p,\varepsilon}\!\left(q^{-\frac{1}{e(p-1)}s_0}+q^{-1}y^{1-s_0}
+q^{-1/2}\left(\frac{q^2}{y}\right)^{1-s_0}\right).
\end{aligned}
\end{equation*}
In particular, as $n \rightarrow \infty$, $L_{\mathrm{av}}(f,\phi)$ converges to $1$ uniformly for $\phi\in\Xi_{\frp}$. More precisely, for $n \to \infty$, we have 
\begin{equation} \label{eq:convergence-average}
\sup_{\phi\in\Xi_{\frp}, \, \cond(\phi)=\frp^n}
\abs{L_{\mathrm{av}}(f,\phi)-1}
\longrightarrow 0
\end{equation}
\end{maintheorem}

\subsection{Previous works}
As noted above, the method of estimating Galois-averaged $L$-functions first appeared in \cite{Rohrlich84,Rohrlich84-2}, where 
\begin{itemize}
    \item In the former work, Rohrlich showed the eventual nonvanishing property of $L$-functions $L(f \otimes \phi, 1)$ where $f$ is a weight two normalized cuspidal newform over $\QQ$ and $\phi$ varies over primitive Dirichlet characters with $p$-power conductor.
    \item In the second work, Rohrlich showed the eventual nonvanishing property of $L$-functions $L(\chi \otimes \phi, 1)$ where $\chi$ is a Hecke character of infinite type $(1,0)$ of root number $+1$ over an imaginary quadratic field $K$ such that $p$ does not divide the class number of $K$, and $\phi$ is a Hecke character over $K$ factoring through the anticyclotomic $\ZZ_p$-extension $K_{\infty}^{\ac}/K$. In a recent work \cite{Jia}, Haijun Jia removed the extra assumption that $p$ does not divide the class number of $K$.
\end{itemize}

In our setting, Jaesung Kwon and Hae-Sang Sun proved in \cite{KS,KwonCorrected} results corresponding to our Theorems \ref{main:A} and \ref{main:B}, under the additional assumption that \emph{$p$ does not divide the class number of $F$}.\footnote{This condition is presented in \cite[Section 4]{KwonCorrected} and hence in the main Theorem 4.1 there, though it does not appear in the introduction and Theorem A of \textit{loc.cit.}.} Inspired by \cite{Jia}, the main innovation of the present article is to remove this extra condition.

\subsection{Method of the proof}
Our proof retains the framework of Kwon \cite{KwonCorrected} but replaces the two steps in Kwon's argument that use the hypothesis $p\nmid h_F$, where $h_F$ is the class number of $F$. Following \cite{KwonCorrected}, the averaged $L$-value $L_{\mathrm{av}}(f, \phi)$  is decomposed into two parts
\[
L_{\mathrm{av}}(f,\phi) = S_{1}(\phi; y) + S_{2}(\phi; y),
\]
where we call the two summands the \emph{ordinary part} and the \emph{Fricke part}, respectively.

We now briefly describe our two modifications. The notation used in this overview is introduced fully in the subsequent sections.

(1) In the estimate of the ordinary part, the \emph{bounded fibre argument} in Section \ref{sec:bounded-fiber} replaces the paragraph immediately preceding \cite[Lemma 3.4]{KwonCorrected}, together with \cite[Lemma 3.4]{KwonCorrected} itself.

Kwon uses $p\nmid h_F$ to assert that $x \mapsto x^{h_F(p-1)}$ is a $(p-1)$-to-one surjection $\calO_{\frp}^{\times}\to 1+\frp\calO_{\frp}$. \cite[Lemma 3.4]{KwonCorrected} then shows that, in each fixed ideal class, the parameters representing ideals in the support of the averaged character $\phi_{\mathrm{av}}$ (see \eqref{eq:ordinary-character-average}) lie in a union of $p-1$ residue classes modulo $\frp^{n-c}$. This feeds into Kwon's \cite[Proposition 3.8]{KwonCorrected} and the estimate of the ordinary half in \cite[Proposition 3.10]{KwonCorrected}.

We instead let $e_F$ denote the exponent of the class group $\Cl(F)$. If an ideal $\fra$ belongs to the support of $\phi_{\mathrm{av}}$ and $\fra^{e_F}=(\beta)$, then we show that $\beta^{p-1}\equiv1\pmod{\frp^{n-c}}$, where $c$ depends only on the fixed coefficient field. After choosing a representative for each ideal class, this becomes a congruence of the form
\[
x^{e_F(p-1)}\equiv u\pmod{\frp^{n-c}}.
\]
Every nonempty fibre of this power map $x \mapsto x^{e_{F}(p-1)}$ has at most ${(p-1)p^{v_p(e_F)}}$ elements, and this is still independent of $n$. This recovers the bounded union of translated residue classes needed for the estimate in the ordinary part.

(2) In the estimate of the Fricke part, the \emph{translated local support argument} in Section \ref{sec:translated-local-support} replaces the proof of \cite[Proposition 3.6]{KwonCorrected}.

In the proof of \cite[Proposition 3.6]{KwonCorrected}, Kwon explicitly chooses $c_n \equiv h_F^{-1}\pmod{p^n}$, writes $\frN^{h_F}=(\beta)$ and uses $c_nh_F\equiv1\pmod{p^n}$ to replace $\phi(\frN)$ by the value $\phi(\beta^{c_n})$ on the principal element. This is impossible when $p\mid h_F$. Our argument retains the idele class involving the level ideal $\frN$ as a multiplier in the Fricke average. If the multiplier has value $\xi$ under $\phi$, the averaged character $\phi_{\mathrm{av}}$ can be nonzero only when
\[
\overline{\phi}_{\frp}(z)\in\xi^{-1}\mu_{p^c}.
\]
The multiplier therefore translates the target of the local character condition without changing its cardinality. Since the target contains at most $p^c$ character values and every such value has $p-1$ preimages, at most $(p-1)p^c$ residue classes $z\bmod p^n$ contribute.

With Kwon's Kloosterman sum estimate in \cite[Lemma 3.5]{KwonCorrected} retained, the resulting analogue of \cite[Proposition 3.6]{KwonCorrected} is then used for the Fricke part estimate, corresponding to \cite[Proposition 3.11]{KwonCorrected}.

We have also reorganized parts of the presentation in \cite{KS,KwonCorrected} to reflect our understanding of the estimates, and we have supplied proofs of several intermediate details that are left implicit there. To make the present article reasonably self-contained, we give the complete argument rather than merely citing the corresponding results of \cite{KS,KwonCorrected}.

\subsection{Arithmetic applications}
In the former work \cite{Rohrlich84}, combined with the Iwasawa theory for  elliptic curves with complex multiplication (\textit{viz.} CM elliptic curves), it is shown in \cite[Corollary on page 409]{Rohrlich84} that for any CM elliptic curve $E/\QQ$ with CM given by a ring of integer of an imaginary quadratic field, the Mordell--Weil group $E(\QQ_{\infty})$ is a finitely generated abelian group, where $\QQ_{\infty}/\QQ$ is the cyclotomic $\ZZ_p$-extension of $\QQ$.

Based on \cite{Rohrlich84-2} and its recent generalization \cite{Jia}, it is shown in \cite{Greenberg, LX} that for elliptic curves $E/\QQ$ with CM by an imaginary quadratic field $K$, the Mordell--Weil group $E(K_{\infty}^{\ac})$ is a finitely generated abelian group if and only if $p$ splits in $K$ and the root number of $E$ is $+1$. Based on \cite{Rohrlich84-2, Jia}, the work \cite{LX} also manages to describe the growth pattern of $\rank_{\ZZ} E(K_{n}^{\ac})$ as $n$ grows, under all reduction types of $p$ in $K$.

Inspired by these arithmetic applications, our main result has the following application. Let $K$ be any imaginary quadratic field such that $p$ splits in $K$ as $p\calO_{K} = \frp\overline{\frp}$. By the main theorem of \cite{QiStokes}, there exists a unique $\frp$-ramified $\ZZ_p$-extension $K_{\infty}^{(\frp)}/K$. Let $E/\QQ$ be an elliptic curve that does not have CM by $K$. In a forthcoming work of the author with Peikai Qi, we establish the Iwasawa theory for such elliptic curves with good ordinary reduction at $p$ along this particular $\ZZ_p$-extension $K_{\infty}^{(\frp)}/K$. With the eventual nonvanishing theorem (Theorem \ref{main:A}) of this article as the most crucial analytic input, we are able to prove that $E(K_{\infty}^{(\frp)})$ is a finitely generated abelian group \footnote{We may impose additional hypotheses on the elliptic curve $E/\QQ$.}. This gives evidence for Mazur's growth conjecture which was first announced in his ICM 1983 report \cite{Mazur}; see also \cite[page 1878]{Lei} for a recent treatment.

We record the corollary of Theorem \ref{main:A} in this setting as follows.
\begin{maintheorem}[{Theorem \ref{thm:elliptic}}]
Let $E/\QQ$ be an elliptic curve of conductor $N$. Suppose $p \nmid N$, $p$ splits in $K$ and $E$ does not have CM by $K$. Then for every finite-order character $\phi$ of $\Gamma_{\frp}$ of sufficiently large conductor, we have $L(E/K,\phi,1) \ne 0$.
\end{maintheorem}

\subsection*{Acknowledgements}
We thank Haijun Jia for helpful conversations during 2022-2024 concerning Rohrlich's work \cite{Rohrlich84-2} and his own related research \cite{Jia}. We also thank our advisor, Professor Xin Wan, for drawing our attention to Rohrlich's work and for suggesting that we consider the problem of removing the hypothesis \(p\nmid h_F\) in the setting of \cite{Rohrlich84-2}.

This work was carried out during the author's postdoctoral appointment at the University of Vienna under the supervision of Professor Harald Grobner and the co-supervision of Professor Adel Betina, to whom we would like to express our sincere gratitude for their support and guidance. This work (together with the author's postdoctoral position) is supported by the \emph{FWF Principal Investigator Project (PAT-2584625): Eigenvarieties and $p$-adic $L$-functions}.
\section{Preliminaries}
\label{sec:ray}

\subsection{Ray class groups}

For $n\geq1$, let $I(\frp)$ be the group of fractional ideals prime to $\frp$, and let $P_1(\frp^n)$ be the subgroup generated by the principal ideals $(\alpha)$ with $v_{\frp}(\alpha)=0$ and $\alpha\equiv1\pmod{\frp^n}$. Let $\Cl(F,\frp^n)=I(\frp)/P_1(\frp^n)$ be the ray class group of $F$ of modulus $\frp^{n}$. It fits into the following exact sequence:
\begin{equation}\label{eq:finite-ray}
 \calO_F^\times\longrightarrow(\calO_F/\frp^n)^\times\longrightarrow\Cl(F,\frp^n)\longrightarrow\Cl(F)\longrightarrow1.
\end{equation}
Taking inverse limits in \eqref{eq:finite-ray} with respect to $n$, the surjectivity of transition maps therefore gives
\begin{equation}\label{eq:infinite-ray}
 1\longrightarrow\overline{\calO_F^\times}\longrightarrow\calO_{F,\frp}^\times\longrightarrow \Cl(F, \frp^{\infty}) \longrightarrow\Cl(F)\longrightarrow1,
\end{equation}
where $\Cl(F, \frp^{\infty})=\varprojlim_n\Cl(F,\frp^n)$ and the bar denotes the closure of the image of the global units.

We denote by $\Xi_{\frp}$ the set of $p$-power order characters of $\Cl(F, \frp^{\infty})$ of $\frp$-power conductor. We assume that the residue degree and the ramification degree of $\frp$ over $p$ are $1$. In this case, $\rank_{\ZZ_p} \Cl(F, \frp^{\infty}) \leq 1$. We only consider the case when $\rank_{\ZZ_p} \Cl(F, \frp^{\infty}) = 1$. If the rank is zero, then $\Xi_{\frp}$ is a finite set, and hence our main theorem is vacuously true.

\subsection{Idelic Hecke characters}

Every character in $\Xi_{\frp}$ can be regarded as a finite-order Hecke character
\[
 \phi:F^\times\backslash\AA_F^\times\longrightarrow\CC^\times
\]
that is unramified away from $\frp$ and trivial at infinity. If $\fra$ is prime to $\frp$, let $\bfa$ be a finite idele with the valuations of $\fra$ away from $\frp$ and with $\frp$-component one, and set $\phi(\fra)=\phi(\bfa)$. If $\iota_{\frp}(u)$ is the idele whose $\frp$-component is $u$ and whose other components are one, put $\phi_{\frp}(u)=\phi(\iota_{\frp}(u))$. We gather the following observations.

\begin{itemize}
    \item Triviality on the diagonal copy of $F^\times$ inside $\AA_{F}^{\times}$ gives
\begin{equation}\label{eq:principal-relation}
 \phi((\alpha))=\phi_{\frp}(\alpha)^{-1} \quad
  \text{ for } v_{\frp}(\alpha)=0, \text{ i.e. } ((\alpha), \frp) = 1.
\end{equation}
    \item Taking $\alpha=u\in\calO_F^\times$ in \eqref{eq:principal-relation}, and using $(u)=\calO_F$ and $\phi(\calO_F)=1$, gives
\begin{equation}\label{eq:global-unit-kernel}
 \phi_{\frp}(u)=1  \quad \text{ for }u\in\calO_F^\times.
\end{equation}
\end{itemize}

Whenever a local residue $a\in(\ZZ/p^n\ZZ)^\times$ occurs inside $\phi$, we mean $\phi_{\frp}(\widetilde a)$ for any lift $\widetilde a\in\ZZ_p^\times$. The value is well defined for a character of conductor $\frp^n$.

We fix $F_{\frp}=\QQ_p$ and $\calO_{F,\frp}=\ZZ_p$. For $j\geq1$, put $U_j=1+p^j\ZZ_p$. Since $p$ is an odd prime, there is a canonical product decomposition
\begin{equation}\label{eq:local-units}
 \ZZ_p^\times=\mu_{p-1}\times U_1,
\end{equation}
and $U_1/U_n$ is cyclic of order $p^{n-1}$.

\begin{lemma}\label{lem:faithful}
Let $\phi$ have $p$-power order and exact conductor $\frp^n$, with $n\geq2$. Then $\phi_{\frp}$ is trivial on $\mu_{p-1}$, and its restriction to $U_1/U_n$ is an isomorphism onto $\mu_{p^{n-1}}$. For every integer $c$ with $0\leq c<n-1$ and every $x\in\ZZ_p^\times$,
\begin{equation}\label{eq:local-order-cutoff}
 \phi_{\frp}(x)^{p^c}=1 \quad\Longleftrightarrow\quad x^{p-1}\equiv1\pmod{p^{n-c}}.
\end{equation}
In particular, if a character of exact conductor $\frp^n$ exists, then for $u\in\calO_F^\times$, we have
\begin{equation}\label{eq:unit-congruence}
 u^{p-1}\equiv1\pmod{\frp^n}.
\end{equation}
\end{lemma}

\begin{proof}
The tame factor $\mu_{p-1}$ has order prime to $p$, so a character of $p$-power order is trivial on it. If the restriction to the cyclic group $U_1/U_n$ were not injective, its kernel would then be a nontrivial subgroup of the cyclic $p$-group $U_{1}/U_{n}$, and hence would contain its unique subgroup $U_{n-1}/U_n$ of order $p$. The character would then be trivial on $U_{n-1}$, contrary to the assumption that $\phi$ has exact conductor $\frp^{n}$. Thus the restriction of $\phi$ to $U_{1}/U_{n}$ is injective and has image $\mu_{p^{n-1}}$.

Write $x=\omega(x)\langle x\rangle$ according to \eqref{eq:local-units} where $\omega(x) \in \mu_{p-1}$ and $\lrangle{x} \in U_{1}$. Because $\phi_{\frp}$ is trivial on $\mu_{p-1}$, we have $\phi_{\frp}(x) = \phi_{\frp}(\lrangle{x})$. Since its restriction to $U_{1}/U_{n}$ is injective, we see that 
\begin{equation} \label{eq:faithful-1}
\phi_{\frp}(x)^{p^{c}} = \phi_{\frp}(\lrangle{x})^{p^{c}} = 1 \Longleftrightarrow \lrangle{x}^{p^{c}} \in U_{n}.
\end{equation}
Note the following equivalences: 
\begin{multline} \label{eq:faithful-2}
\lrangle{x}^{p^{c}} \in U_{n} \Longleftrightarrow \log_{p}\lrangle{x} \in p^{n-c} \ZZ_p \Longleftrightarrow \lrangle{x} \in U_{n-c} \\ \Longleftrightarrow \lrangle{x}^{p-1} \in U_{n-c} \Longleftrightarrow x^{p-1} \in 1 + p^{n-c} \ZZ_p,
\end{multline}
where the third equivalence uses the fact that $p-1 \in \ZZ_p^{\times}$, and hence after applying $\log_p:U_1\cong p\ZZ_p$, multiplication by $p-1$ preserves every subgroup $p^{j}\ZZ_p$; and the fourth equivalence uses $x^{p-1} = \lrangle{x}^{p-1}$. Combining equations \eqref{eq:faithful-1} and \eqref{eq:faithful-2} gives \eqref{eq:local-order-cutoff}. For the final assertion, equation \eqref{eq:global-unit-kernel} gives $\phi_{\frp}(u)=1$. Applying \eqref{eq:local-order-cutoff} with $c=0$ then gives the result.
\end{proof}

\subsection{The partial Gauss sum and the average Gauss sum}
We define $\bfe(z) := \exp(2 \pi \rmi z)$ for $z \in \CC$. Let $\bfe_{F}$ be the character of $F \bs \AA_{F}$ characterized by $\bfe_{F}(z)=\bfe(\Tr_{F/\QQ}(z))$ for $z \in F_{\RR} := F \otimes_{\QQ} \RR$.

Let $\QQ_f$ be the Hecke field of $f$ over $F$. We put $K_0=\QQ_f(\mu_p)$, and choose the unique integer $c\geq1$ such that
\begin{equation}\label{eq:cyclotomic-intersection-intro}
 K_0\cap\QQ(\mu_{p^\infty})=\QQ(\mu_{p^c}).
\end{equation}
If $\phi\in\Xi_{\frp}$ has conductor $\frp^n$ and $n > 2c$, its restriction to $1+p\ZZ_p$ has exact order $p^{n-1}$. The map $b\mapsto\phi_{\frp}(1+p^{n-c}b)$ is a primitive additive character of $\ZZ/p^c\ZZ$. Therefore, there is a unique $c_\phi\in(\ZZ/p^c\ZZ)^\times$ such that 
\footnote{For $M=n-c$, we have
\[
(1+p^M b)(1+p^M b^\prime)\equiv1+p^M(b+b^\prime)\pmod{p^n}
\]
when $2M\geq n$. Thus $n\geq2c$ makes $b\mapsto\phi_{\frp}(1+p^{n-c}b)$ additive on $\ZZ/p^c\ZZ$; its primitivity follows from the fact that $\phi$ has exact conductor $\frp^n$. The corresponding assertion in \cite[page 126]{KwonCorrected}, however, is stated under the weaker condition $n>c$ (note that the notation $n_{0}$ in \cite{KwonCorrected} is our $c$), which seems insufficient. Nevertheless, this correction does not affect the eventual
estimates, which concern sufficiently large $n$.}
\begin{equation}\label{eq:cphi-intro}
 \phi_{\frp}(1+p^{n-c}b)=\exp(2\pi\rmi b c_\phi/p^c)
 \qquad(b\in\ZZ_p).
\end{equation}

Let $G_\phi=\Gal(\QQ(\phi)/\QQ(\mu_{p^c}))$. Put $\frc=\frp^n$, $q=p^n$, and $\rme_q(a)=\exp(2\pi\rmi a/q)$. Let $\varpi_{\frc}$ be the finite idele with $\frp$-component $p^n$ and all other components $1$. As in \cite[Section 2.5]{KwonCorrected}, let $\calR_{\frc}$ be the image in $\AA_F$ of a complete set of representatives of $(\frc^{-1}/\calO_F)^\times$, embedded at $\frp$ and set equal to zero at the other finite places.

Since $\frp\nmid\frd_F$, choose a finite idele $d_F$ representing $\frd_F$ with $(d_F)_{\frp}=1$. Define the \emph{full Gauss sums}
\begin{equation}\label{eq:global-gauss}
 \calG(\phi):=\phi(d_F)^{-1} \sum_{u\in\calR_{\frc}} \phi_{\frp}(p^n u_{\frp})\bfe_F(u/d_F), \quad \calG_0(\phi):= \sum_{a\in(\ZZ/q\ZZ)^\times} \phi_{\frp}(a)\rme_q(a).
\end{equation}
The map $u\mapsto p^n u_{\frp}$ identifies $\calR_{\frc}$ with $(\ZZ/q\ZZ)^\times$. Our global additive character convention gives $\bfe_{F,\frp}(a/p^n)=\rme_q(-a)$. Consequently,
\[
 \sum_{u\in\calR_{\frc}} \phi_{\frp}(p^n u_{\frp})\bfe_F(u/d_F) = \sum_{a\in(\ZZ/q\ZZ)^\times} \phi_{\frp}(a)\rme_q(-a) =\calG_0(\phi).
\]
The last equality follows by replacing $a$ with $-a$, since $\phi_{\frp}(-1)=1$. Moreover, $\phi(d_F)=\phi(\frd_F)$, because $(d_F)_{\frp}=1$ and $\phi$ is unramified away from $\frp$. We have therefore proved \footnote{In \cite{KwonCorrected}, the author identified $\calG(\phi)$ with $\calG_{0}(\phi)$ in his estimate, which may not hold since $\phi(d_{F})$ might not be trivial. In this article, we modify the definitions of the averaged Gauss sums using $\calG_{0}(\phi)$ instead. The Gauss sum $\calG(\phi)$ also plays an important role, since it is the Gauss sum used in the automorphic setting of \cite{Hida}.}
\begin{equation}\label{eq:gauss-bridge}
 \calG(\phi)=\phi(\frd_F)^{-1}\calG_0(\phi).
\end{equation}

In \cite[Section 3.1]{KwonCorrected}, Kwon also defined the \emph{partial Gauss sum} for $\sigma \in G_{\phi}$ as
\begin{equation}\label{eq:partial-factor}
\calP_\sigma(\phi)=\sum_{\substack{a \bmod q\\a \equiv c_\phi\, (p^c)}}
 \overline\phi^{\,\sigma}(a)\rme_q(a).
\end{equation}
Following \cite[Definition 3.2]{KwonCorrected}, we define the averages of Gauss sums.

\begin{definition} \label{def:average_Gauss_sums}
For an ideal $\fra$ prime to $\frp$, we define the \emph{averaged Gauss sums} as follows:
\begin{equation}\label{eq:Gav-definition}
 \calG_{\mathrm{av}}(\phi,\fra)=\frac{1}{\abs{G_\phi}q}\sum_{\sigma\in G_\phi}
 \calP_\sigma(\phi) \calG_0(\phi^\sigma)\phi^\sigma(\fra),
\end{equation}
and
\begin{equation}\label{eq:Giota-definition}
 \calG_{\mathrm{av}}^{\iota}(\phi,\fra)=\frac{1}{\abs{G_\phi}q}
 \sum_{\sigma\in G_\phi} \calP_\sigma(\phi) \calG_0(\phi^\sigma)W(\phi^\sigma)
 \overline\phi^{\,\sigma}(\fra),
\end{equation}
where $W(\phi) = q^{-1} \calG(\bar\phi)^2\phi(\varpi_{\frp^{2n}\frN})$. The averaged Gauss sums are zero by convention when $\frp\mid\fra$.    
\end{definition}

\section{Automorphic backgrounds and average $L$-values}
\label{sec:automorphic}

In this section, we review the theory of cusp forms on $\GL(2)$ over $F$ and their $L$-functions, following \cite[Section 2]{KwonCorrected}.

\subsection{Cusp forms over $\GL(2)$ and their automorphic $L$-functions}

Let $\Sigma(\RR)$ and $\Sigma(\CC)$ be the sets of real and complex places of $F$, respectively, and let $\sfc$ denote complex conjugation. Let $I_F$ be the set of embeddings $F\hookrightarrow\CC$. A \emph{cohomological weight} is a pair $(\bfk, \bfm)$ with
\[
 \bfk=\sum_{\tau\in I_F}k_\tau\tau \in \ZZ[I_{F}],
 \qquad
 \bfm=\sum_{\tau\in I_F}m_\tau\tau \in \ZZ[I_{F}],
\]
such that $k_\tau\geq2$, $k_\tau=k_{\sfc\tau}$ and $m_{\tau} \geq 0$ for any $\tau \in I_{F}$, and $k_\tau+2m_\tau=:k$ is independent of $\tau \in I_{F}$. 

We denote $L_{\bfk} := \otimes_{\tau \in \Sigma(\CC)}\Sym^{2k_{\tau}-2}(\CC^{2})$ and $C_{F,\infty}^{+} := \SO_2(\RR)^{\Sigma(\RR)} \times \mathrm{SU}_{2}(\CC)^{\Sigma(\CC)}$ which acts on $L_{\bfk}$. Let $\frN$ be a nonzero integral ideal of $F$. We define an open subgroup $U_0(\frN)\subseteq\GL_2(\widehat{\calO}_F)$ by 
\[
U_0(\frN) := \left\{ \begin{bmatrix}
a & b \\ c & d
\end{bmatrix} \in \GL_2(\widehat{\calO}_F) : c \in \widehat{\frN}  \right\},
\]
where $\widehat{A} := A \otimes_{\ZZ} \hatZZ$ for any $\ZZ$-subalgebra $A$ of $F$. Let $\chi$ be a Hecke character modulo $\frN$ with infinite type $(2-k)\sum_{\tau \in I_{F}}\tau \in \ZZ[I_{F}]$. Let $J$ be a subset of $\Sigma(\RR)$.

\begin{definition} \label{def:cusp-form}
A cusp form on $\GL_2(\AA_F)$ of weight $(\bfk, \bfm)$, archimedean type $J$, level $\frN$ and central character $\chi$ is a smooth function $f: \GL_2(\AA_{F}) \rightarrow L_{\bfk}$ such that 
\begin{enumerate}
    \item $D_{\tau}f = \frac{1}{2} k_{\tau} (k_{\tau} - 2)f$ for the Casimir operators $\{D_{\tau} : \tau \in I_{F}\}$ of $\SL_2$,
    \item For elements $\gamma \in \GL_2(F)$, $z \in \AA_{F}^{\times}$, $g \in \GL_2(\AA_{F})$, $u = \begin{bmatrix} a & b \\ c & d \end{bmatrix} \in U_{0}(\frN)$, and $c \in C_{F,\infty}^{+}$, $f$ satisfies the following:
    \[
    f(\gamma z g u c)(\bfx) = \chi(z) \chi\left( \prod_{v \mid \frN} d_{v} \right) \bfe \left( \sum_{\tau \in J} k_{\tau}\theta_{\tau} - \sum_{\tau \in \Sigma(\RR) \smallsetminus J} k_{\tau}\theta_{\tau} \right) f(g)(c_{\tau} \bfx_{\tau})_{\tau \in \Sigma(\CC)},
    \]
    where $c_{\tau} = \begin{bmatrix}
        \cos(2\pi \theta_{\tau}) & -\sin(2\pi \theta_{\tau}) \\ \sin(2\pi \theta_{\tau}) & \cos(2\pi \theta_{\tau})
    \end{bmatrix}$ for $\tau \in \Sigma(\RR)$,
    \item Let $\dif x$ be a Haar measure on $F \bs \AA_{F}$. For $g \in \GL_2(\AA_{F})$, 
    \[
    \int_{F \bs \AA_{F}} f \left( \begin{bmatrix}
        1 & x \\ 0 & 1
    \end{bmatrix} g \right) \dif x = 0.
    \]
\end{enumerate}
From now on, we denote by $S_{(\bfk,\bfm),J}(\frN,\chi)$ the space of such cusp forms $f$. Moreover, 
\begin{enumerate}[label = \rm (\arabic*)]
    \item For $f$ and an ideal $\fra$ of $F$, we denote by $a_{f}(\fra)$ the Fourier--Whittaker coefficient of $f$ at $\fra$.
    \item A cusp form $f$ on $\GL(2)$ is called a Hecke eigenform if $f$ is a common eigenfunction of all the Hecke operators.
    \item A Hecke eigenform is said to be normalised if $a_{f}(\calO_{F}) = 1$.
\end{enumerate}
\end{definition}

Throughout this article, we assume that $f \in S_{(\bfk,\bfm),J}(\frN,\chi)$ is a normalized cuspidal Hecke eigenform on $\GL(2)$ over $F$.

\begin{definition}
The \emph{automorphic $L$-function} of $f$ is defined as
\begin{equation}\label{eq:arithmetic-L-series}
 L(s,f)=\sum_{0\ne\fra\subseteq\calO_F}\frac{a_f(\fra)}{(\Norm\fra)^s}.
\end{equation}   
\end{definition}

Following \cite[Section 2.3]{KwonCorrected}, we define the gamma factor as
\begin{equation*}\label{eq:corrected-gamma}
 \Gamma(s,f)=[F_J^\times:F_+^\times]\left(\frac{\pi}{2}\right)^{\abs{\Sigma(\CC)}}
 \prod_{\tau\in I_F}\frac{\Gamma(s-m_\tau)}{(2\pi)^{s-m_\tau}}.
\end{equation*}
Then the completed $L$-function is defined as
\begin{equation*}\label{eq:completed-L}
 \Lambda(s,f) := D_F^s(\Norm\frN)^{s/2}\Gamma(s,f)L(s,f).
\end{equation*}

We denote $\varpi_{\frc} := \prod_{v \mid \frc} \varpi_{v}^{\ord_{v}(\frc)}$ for any integral ideal $\frc$ of $F$.   The Fricke operator is defined as 
\begin{equation*}\label{eq:corrected-fricke}
 W_{\frN}f(g)=\rmi^{\sum_{\tau\in\Sigma(\RR)}k_\tau}(\Norm\frN)^{1-k/2}
 \overline{\omega}_\chi(\det g)
 f\!\left(g\begin{pmatrix}0&-1\\ \varpi_{\frN}&0\end{pmatrix}\right),
\end{equation*}
where $\omega_{\chi} := \chi \abs{-}^{k-2}$ is a unitarization of $\chi$. Then \cite[Proposition 2.9]{KwonCorrected} proves the following functional equation for $\Lambda(s,f)$:

\begin{proposition}
The completed $L$-function $\Lambda(s,f)$ extends to an entire function of $s \in \CC$, and
\begin{equation}\label{eq:corrected-functional-equation}
 \Lambda(s,f)=W_f\Lambda(k-s,W_{\frN}f),
\end{equation}
where $W_f=(-1)^{\sum_{\tau\in\Sigma(\RR)\setminus J}k_\tau}
 (-1)^{\sum_{\tau\in\Sigma(\CC)}k_\tau}$.
\end{proposition}

\subsection{Galois-averaged $L$-function and its decomposition}

The goal of this subsection is to define the Galois-averaged $L$-function and establish its decomposition \eqref{eq:averaged-afe}.

\subsubsection{Mellin transform}
Let $\Psi$ be a compactly supported smooth function on $\RR_{>0}$ such that \footnote{There seems to be a typographical error in \cite[page 124]{KwonCorrected}, where the normalisation of $\Psi$ was written as $\int_{0}^{\infty} \Psi(y) \dif y = 1$. We have corrected it here.} $\int_{0}^{\infty} \Psi(y)/y \dif y = 1$, and put
\[
 \Phi(z) := \int_0^\infty\Psi(u)u^z\frac{\dif u}{u}.
\]
For $j=1,2$, define
\begin{equation}\label{eq:V-definitions}
 V_{j,s}(x)=\frac{1}{2\pi\rmi}\int_{2 + \rmi \RR}\Phi((-1)^{j-1}z)\Gamma(s+z,f)x^{-z}\frac{\dif z}{z}.
\end{equation}
It is shown in \cite[Equation (2.3)]{KwonCorrected} that for any fixed $A>0$,
\begin{equation}\label{eq:kernel-bounds}
 V_{j,k/2}(x)\ll_A(1+x)^{-A},
 \text{ and }
 V_{j,k/2}(x)=\Gamma(k/2,f)+O(x^{1/2}) \text{ for } 0 < x \leq 1.
\end{equation}

\subsubsection{Twisted cusp forms}
Let $\phi \in \Xi_{\frp}$ be Hecke character of exact conductor $\frp^{n}$. In \cite[Section 2.5]{KwonCorrected}, Kwon defined the \emph{twisted cusp form}
\[
(f \otimes \phi) (g) := \dfrac{\phi(\det g)}{\calG(\phi)} \sum_{u \in \calR_{\frc}} \phi_{\frp}(p^{n}u_{\frp}) f\left( g \begin{bmatrix}
    1 & u \\ 0 & 1
\end{bmatrix} \right). 
\]
Then the twisted cusp form $f\otimes\phi\in S_{(\bfk,\bfm),J}(\frN\frc^2,\chi\phi^2)$ satisfies $a_{f\otimes\phi}(\fra)=a_f(\fra)\phi(\fra)$ by \cite[Section 6]{Hida}. It is proved in \cite[Proposition 2.10]{KwonCorrected} that 
\begin{equation}\label{eq:corrected-fricke-twist}
 W_{\frN\frc^2}(f\otimes\phi)
 =\omega_\chi(\varpi_{\frc})W(\phi)(W_{\frN}f)\otimes\bar\phi,
\end{equation}
where
\begin{equation}\label{eq:corrected-root-factor}
 W(\phi)=(\Norm\frc)^{-1} \calG(\bar\phi)^2\phi(\varpi_{\frc^2\frN}),
 \text{ with } \abs{W(\phi)}=1.
\end{equation}

\subsubsection{Decomposition of $L$-functions and average $L$-functions}
As shown in \cite[Equation (2.4)]{KwonCorrected}, applying the functional equation \eqref{eq:corrected-functional-equation} to the Mellin integral gives, at the central point $s = k/2$ and for every $y>0$,
\begin{multline}\label{eq:corrected-afe}
 \Gamma(k/2,f)L(k/2,f\otimes\phi) = \sum_{\fra}\frac{a_f(\fra)\phi(\fra)}{(\Norm\fra)^{k/2}}
 V_{1,k/2}\!\left(\frac{\Norm\fra}{y}\right) \\
 +\omega_\chi(\varpi_{\frc})W_fW(\phi)
 \sum_{\fra}\frac{a_{W_{\frN}f}(\fra)\bar\phi(\fra)}{(\Norm\fra)^{k/2}}
 V_{2,k/2}\!\left(\frac{\Norm\fra\,y}{\Norm(\frN\frc^2)}\right).
\end{multline}
Both sums run over nonzero integral ideals and converge absolutely because of \eqref{eq:kernel-bounds}.

\begin{definition} \label{def:averaged-L-function}
Following \cite[Definition 3.9]{KwonCorrected} \footnote{As remarked before, we replace the appearance of the Gauss sum $\calG(\phi)$ in \cite{KwonCorrected} by $\calG_{0}(\phi)$ here.}, we define the \emph{Galois-averaged $L$-function} of $f$ twisted by $\phi$ as
\begin{equation}\label{eq:additive-average-intro}
 L_{\mathrm{av}}(f,\phi) =\frac{1}{\abs{G_\phi}q} \sum_{\substack{a\bmod q\\a\equiv c_\phi\, \bmod{p^c}}} \rme_q(a) \sum_{\sigma\in G_\phi} \overline{\phi}^{\,\sigma}(a) \calG_{0}(\phi^\sigma) L\!\left(\frac{k}{2},f\otimes\phi^\sigma\right).
\end{equation}
\end{definition}

Together with \eqref{eq:corrected-afe}, the definitions of $\calG_{\mathrm{av}}$ and $\calG_{\mathrm{av}}^{\iota}$ in Definition \ref{def:average_Gauss_sums} give the following decomposition
\begin{equation}\label{eq:averaged-afe}
 L_{\mathrm{av}}(f,\phi)=S_1(\phi;y)+S_2(\phi;y),
\end{equation}
where
\begin{equation}
    S_1(\phi;y) =\frac{1}{\Gamma(k/2,f)}\sum_{\fra} \frac{a_f(\fra)\calG_{\mathrm{av}}(\phi,\fra)}{(\Norm\fra)^{k/2}} V_{1,k/2}\!\left(\frac{\Norm\fra}{y}\right),\label{eq:S1}
\end{equation}
called the \emph{ordinary part} of $L_{\mathrm{av}}(f,\phi)$, and
\begin{equation}
 S_2(\phi;y) =\frac{\omega_\chi(\varpi_{\frp^n})W_f}{\Gamma(k/2,f)}\sum_{\fra} \frac{a_{W_{\frN}f}(\fra)\calG_{\mathrm{av}}^{\iota}(\phi,\fra)}{(\Norm\fra)^{k/2}} V_{2,k/2}\!\left(\frac{\Norm\fra\,y}{\Norm\frN\,q^2}\right),
 \label{eq:S2}
\end{equation}
called the \emph{Fricke part} of $L_{\mathrm{av}}(f,\phi)$. Both sums run over nonzero integral ideals of $F$. The estimate for the Galois-averaged $L$-function $L_{\mathrm{av}}(f, \phi)$ reduces to estimates for its ordinary and Fricke parts, which we provide in Sections \ref{sec:estimate-ordinary} and \ref{sec:estimate-fricke} in that order.

\subsection{Fourier-coefficient bounds}
We assume that there exists an exponent $0\leq\theta<1/4$ such that, for every $\varepsilon>0$,
\begin{equation}\label{eq:central-coefficient-bound-auto}
 \frac{\abs{a_f(\fra)}+\abs{a_{W_{\frN}f}(\fra)}}{(\Norm\fra)^{k/2}}
 \ll_{F,f,\varepsilon}
 (\Norm\fra)^{-1/2+\theta+\varepsilon}
\end{equation}
for every nonzero integral ideal $\fra$. The known unconditional estimates allow one to take $\theta=7/64$ (see \cite{BB,Nakasuji}), which is sufficient for our argument; we fix this choice of $\theta$.

Fix $\varepsilon \in (0, \frac{1}{4} - \theta)$ and define $t=\theta+\varepsilon$, $s_0=\frac{1}{2}-t$. Then
\begin{equation} \label{eq:range-t}
    0<t<\frac{1}{4}, \qquad \frac{1}{4}<s_0<\frac{1}{2}. 
\end{equation}
With this notation, \eqref{eq:central-coefficient-bound-auto} gives
\begin{equation}\label{eq:central-coefficient-bound}
\frac{\abs{a_f(\fra)}}{(\Norm\fra)^{k/2}}\leq \frac{\abs{a_f(\fra)}+\abs{a_{W_{\frN}f}(\fra)}}{(\Norm\fra)^{k/2}} \ll_{F,f,\varepsilon}(\Norm\fra)^{-s_0}.
\end{equation}
Similarly,
\begin{equation}\label{eq:central-coefficient-bound2}
\frac{\abs{a_{W_{\frN}f}(\fra)}}{(\Norm\fra)^{k/2}} \leq \frac{\abs{a_f(\fra)}+\abs{a_{W_{\frN}f}(\fra)}}{(\Norm\fra)^{k/2}} \ll_{F,f,\varepsilon}(\Norm\fra)^{-s_0}.
\end{equation}
\section{Estimate on the ordinary part}
\label{sec:estimate-ordinary}

Fix $c$ by \eqref{eq:cyclotomic-intersection-intro}, let $\phi\in\Xi_{\frp}$ have conductor $\frp^n$, and assume $n>2c$. Put $q=p^n$ and $G_\phi=\Gal(\QQ(\phi)/\QQ(\mu_{p^c}))$.

\subsection{Averaged characters and Gauss sums}

For an idele class or an ideal $\fra$ prime to $\frp$, define
\begin{equation}\label{eq:ordinary-character-average}
 \phi_{\mathrm{av}}(\fra)=\frac{1}{\abs{G_\phi}}\sum_{\sigma\in G_\phi}\phi^\sigma(\fra) = \dfrac{1}{[\QQ(\phi):\QQ(\mu_{p^{c}})]} \Tr_{\QQ(\phi)/\QQ(\mu_{p^{c}})}(\phi(\fra)).
\end{equation}

\begin{lemma}\label{lem:trace}
For every $\fra$ as above, we have
\begin{equation}\label{eq:trace-formula}
 \phi_{\mathrm{av}}(\fra)=
 \begin{cases}
 \phi(\fra),&\phi(\fra)\in\mu_{p^c},\\
 0,&\phi(\fra)\notin\mu_{p^c}.
 \end{cases}
\end{equation}
In particular, $\abs{\phi_{\mathrm{av}}(\fra)}\leq1$ and $\phi_{\mathrm{av}}(1)=1$.
\end{lemma}

\begin{proof}
Let $p^r$ be the exact order of $\phi$. Then $\QQ(\phi) = \QQ(\mu_{p^{r}})$. If $\phi$ has conductor $\frp^n$, Lemma \ref{lem:faithful} says that its restriction to $U_1/U_n=(1+p\ZZ_p)/(1+p^n\ZZ_p)$ is injective. This group has order $p^{n-1}$, so the image of $\phi$ contains $\mu_{p^{n-1}}$. Hence $r\geq n-1$. Under the standing assumption $n>c+1$, we have $r>c$. As a result, 
\[
G_{\phi} = \Gal(\QQ(\mu_{p^{r}})/\QQ(\mu_{p^{c}})) \simeq \{b\in(\ZZ/p^r\ZZ)^\times:b\equiv1\pmod{p^c}\}
\]
via the identification $b \mapsto [\sigma_{b}: \zeta_{p^{r}} \mapsto \zeta_{p^{r}}^{b}]$. Every such $b$ can be written as $1+p^c j$ for $0\leq j<p^{r-c}$, so we also write $\sigma_{j} := \sigma_{1+p^{c}j}$.

Write $\phi(\fra)=\zeta_{p^r}^a$, then $\phi^{\sigma_j}(\fra) =\sigma_j(\phi(\fra)) =\sigma_j(\zeta_{p^r}^{a}) =\zeta_{p^r}^{\,a(1+p^cj)}$, and hence
\[
\phi_{\mathrm{av}}(\fra) = \frac{1}{p^{r-c}}\sum_{j=0}^{p^{r-c}-1} \zeta_{p^r}^{a(1+p^cj)} = \zeta_{p^r}^{a} \frac{1}{p^{r-c}} \sum_{j=0}^{p^{r-c}-1} \left(\zeta_{p^r}^{\,ap^c}\right)^j = \zeta_{p^r}^{a} \frac{1}{p^{r-c}}
\sum_{j=0}^{p^{r-c}-1}
\zeta_{p^{r-c}}^{aj}.
\]
If $\zeta_{p^{r-c}}^{a} = 1$ (which is equivalent to $\phi(\fra) \in \mu_{p^{c}}$), then every summand equals one, and hence $\phi_{\mathrm{av}}(\fra) = \zeta_{p^{r}}^{a}$. Otherwise (i.e. $\phi(\fra) \not\in \mu_{p^{c}}$), the sum over $j$ is zero by the geometric series identity because $(\zeta_{p^{r-c}}^{a})^{p^{r-c}} = 1$. The lemma then follows.
\end{proof}

The following lemma appears as \cite[Lemma 3.3]{KwonCorrected}.
\begin{lemma}\label{lem:partial-gauss-identity}
For every ideal $\fra$,
\begin{equation}\label{eq:Gav-equals-trace}
 \calG_{\mathrm{av}}(\phi,\fra)=\phi_{\mathrm{av}}(\fra).
\end{equation}
Both sides are understood to be zero when $\frp\mid\fra$.
\end{lemma}

\begin{proof}
Assume $(\fra,\frp)=1$. Expanding the definitions gives
\begin{align*}
 \calG_{\mathrm{av}}(\phi,\fra) =\frac{1}{q}\sum_{a_1,a}\phi_{\mathrm{av}}(a_1^{-1}a\fra)\rme_q(a_1+a) =\frac{1}{q}\sum_{a_1,a_2}\phi_{\mathrm{av}}(a_2\fra)\rme_q(a_1(1+a_2)),
\end{align*}
where $a_1\equiv c_\phi\pmod{p^c}$ and $a,a_2$ range over $(\ZZ/q\ZZ)^\times$. Write $a_1=c_\phi+p^cb$. Orthogonality in $b\bmod p^{n-c}$ gives
\begin{equation}\label{eq:partial-orthogonality}
 \calG_{\mathrm{av}}(\phi,\fra)=\frac{1}{p^c}
 \sum_{a_2\equiv-1\, \bmod p^{n-c}}
 \phi_{\mathrm{av}}(a_2\fra)\rme_q(c_\phi(1+a_2)).
\end{equation}
For such $a_2$, equation \eqref{eq:cphi-intro} and the triviality of a $p$-power character on $-1$ give
\[
 \phi(a_2)=\rme_q(-c_\phi(1+a_2)).
\]
This value belongs to $\mu_{p^c}$ and is fixed by $G_\phi$, so $\phi_{\mathrm{av}}(a_2\fra)=\phi(a_2)\phi_{\mathrm{av}}(\fra)$. The two $\rme_{q}$-factors cancel in \eqref{eq:partial-orthogonality}. There are exactly $p^c$ possible $a_2$, and \eqref{eq:Gav-equals-trace} follows.
\end{proof}

\subsection{Bounded fiber arguments} \label{sec:bounded-fiber}

Let $e=e_F$ be the exponent of the class group $\Cl(F)$ of $F$. Choose integral ideals $\frc_1,\ldots,\frc_{h_F}$, prime to $\frp$, such that $\{\frc_i^{-1}\}$ is a complete set of ideal-class representatives and $\frc_1=\calO_F$. Choose $b_i\in\calO_F$ satisfying
\begin{equation}\label{eq:class-generators}
 \frc_i^e=(b_i).
\end{equation}
Every integral ideal $\fra$ in the class of $\frc_i^{-1}$ has a presentation
\begin{equation}\label{eq:ideal-parametrization}
 \fra=(\alpha)\frc_i^{-1},
 \qquad \alpha\in\frc_i,
\end{equation}
unique up to multiplication of $\alpha$ by a global unit by the exact sequence \eqref{eq:infinite-ray}. If $(\fra,\frp)=1$, then $\alpha$ and $b_i$ are units in $F_{\frp}$.

\begin{proposition}\label{prop:power-support}
Suppose $\phi_{\mathrm{av}}(\fra)\ne0$. If $\fra^e=(\beta)$, then
\begin{equation}\label{eq:beta-principal-congruence}
 \beta^{p-1}\equiv1\pmod{\frp^{n-c}}.
\end{equation}
If $\fra=(\alpha)\frc_i^{-1}$ as in \eqref{eq:ideal-parametrization}, then
\begin{equation}\label{eq:alpha-class-congruence}
 \alpha^{e(p-1)}\equiv b_i^{p-1}\pmod{\frp^{n-c}}.
\end{equation}
\end{proposition}

\begin{proof}
Lemma~\ref{lem:trace} gives $\phi(\fra)\in\mu_{p^c}$. By our convention, $(\fra, \frp) = 1$, and hence $v_{\frp}(\beta) = 0$. Since $(\beta)=\fra^e$, equation \eqref{eq:principal-relation} then gives
\[
 \phi_{\frp}(\beta)^{-1}=\phi((\beta))=\phi(\fra)^e\in\mu_{p^c}.
\]
Applying Lemma~\ref{lem:faithful} proves \eqref{eq:beta-principal-congruence}. In the presentation \eqref{eq:ideal-parametrization}, equation \eqref{eq:class-generators} gives $\fra^e=(\alpha^e/b_i)$. Substitution into \eqref{eq:beta-principal-congruence} gives \eqref{eq:alpha-class-congruence}, as desired.
\end{proof}

\begin{lemma}\label{lem:power-fibres}
For $M\geq1$ and $u\in(\ZZ/p^M\ZZ)^\times$, the congruence
\[
 x^{e(p-1)}\equiv u\pmod{p^M}
\]
has either no solution or
\[
 \gcd(e(p-1),(p-1)p^{M-1})=(p-1)p^{\min\{v_p(e),M-1\}} \leq C_{e,p} := (p-1)p^{v_p(e)}
\]
solutions. In particular, the upper bound $C_{e,p}$ is independent of $M$ and $u$.
\end{lemma}

\begin{proof}
For odd $p$, the group $(\ZZ/p^M\ZZ)^\times$ is cyclic of order $(p-1)p^{M-1}$. Every nonempty fibre of the homomorphism $x\mapsto x^D$ is a coset of its kernel, whose order is the displayed greatest common divisor above.
\end{proof}

\subsection{Coherent $\ZZ$-cones}

Let $F_{\RR}=F\otimes_\QQ\RR$. A \emph{$\ZZ$-cone} is a set
\[
 \ZZ_{+}\lrangle{V}=\ZZ_{>0}\langle v_1,\ldots,v_d\rangle \subseteq F_{\RR}^{\times},
 \qquad d=[F:\QQ],
\]
where $v_1,\ldots,v_d$ form a $\QQ$-basis $V$ of $F$. A $\ZZ$-cone $\calB$ is said to be \emph{coherent} if, for every embedding $\tau$, one may choose $\rho_\tau\in\{\operatorname{Re}\tau,\operatorname{Im}\tau\}$ such that all $\rho_\tau(v_j)$ are nonzero and have the same sign.

Let $C_{F}^{0}$ be a fundamental domain of $F_{\RR}^{\times}/\calO_{F}^{\times}$, where $\calO_{F}^{\times}$ acts on $F_{\RR}^{\times}$ via the natural embedding $F^{\times} \hookrightarrow F_{\RR}^{\times}$. We denote by $\frB$ a finite set of coherent $\ZZ$-cones which covers $\calO_{F} \cap C_{F}^{0}$. The existence of such $\frB$ is guaranteed by \cite[Proposition 4]{Rohrlich89}. We fix such data once and for all.

If $\alpha=\sum_jn_jv_j$ belongs to a fixed coherent cone $\calB$, then for every $\tau$, we have
\begin{equation}\label{eq:coherent-coordinate}
\abs{\tau(\alpha)} \geq\abs{\rho_\tau(\alpha)} =\sum_jn_j\abs{\rho_\tau(v_j)}\gg_{\calB,\tau}\max_jn_j.
\end{equation}
The reverse inequality $\abs{\tau(\alpha)}\ll_{\calB,\tau}\max_jn_j$ is immediate:
\[
\abs{\tau(\alpha)} \leq \sum_j n_j \abs{\tau(v_j)} \leq \max_jn_j \sum_j \abs{\tau(v_j)}
\ll_{\calB,\tau} \max_jn_j.
\]
Thus on a fixed coherent cone $\calB$, we have $\abs{\tau(\alpha)} \asymp_{\calB,\tau} \max_jn_j$. Additionally, it follows from the definition that
\[
\abs{N_{F/\QQ}(\alpha)} \gg_{\calB, \tau} \left( \max_{j} n_{j} \right)^{d}.
\]
Since $\frB$ is a finite set, it follows that for $\alpha \in \calO_{F} \cap C_{F}^{0}$, we have
\begin{equation}\label{eq:balanced-domain}
 \max_{\tau:F\hookrightarrow\CC}\abs{\tau(\alpha)}
 \ll_F\abs{N_{F/\QQ}(\alpha)}^{1/d}.
\end{equation}

We are now ready to prove the following estimate. It is a translated version of the estimate from Rohrlich's \cite[Proposition~5]{Rohrlich89} that appears on \cite[page 131]{KwonCorrected}.

\begin{proposition} \label{prop:cone-progression}
Let $\frc$ be a fixed integral ideal prime to $\frp$. For $M\geq1$, let $\calR = r + \frp^{M} \frc$ be a coset of $\frp^M\frc$ in $\frc$. Uniformly in $\calR$, $M$, and $X\geq1$, we have
\begin{equation}\label{eq:translated-progression}
 \#\{\alpha\in\calR\cap C_F^0:0<\abs{N_{F/\QQ}(\alpha)}\leq X\}
 \ll_{F,\frc}1+\frac{X}{p^M}.
\end{equation}
\end{proposition}

\begin{proof}
Embed $F$ into its $d$-dimensional $\RR$-vector space $F_{\RR}$ via $\iota: F \hookrightarrow F_{\RR}$. For every point $\alpha$ counted on the left of \eqref{eq:translated-progression}, its Euclidean norm $\aabs{\iota(\alpha)}$ satisfies $\aabs{\iota(\alpha)} \ll_{d} \max_{\tau}{\abs{\tau(\alpha)}}$. By \eqref{eq:balanced-domain}, $\alpha$ lies in a Euclidean ball $\sfB(0, C_{F} X^{1/d})$ where $C_{F}$ is a constant depending only on $F$.

Take distinct $\alpha_1, \alpha_2$ in the counted set and set $0 \ne x := \alpha_1 - \alpha_2 \in\frp^M\frc$. Then 
\[
 \abs{N_{F/\QQ}(x)}\geq\Norm(\frp^M \frc) = p^{M} \Norm(\frc) \gg_{F,\frc}p^M.
\]
To bound its Euclidean norm, note that $\aabs{\iota(x)} \gg_{F} \abs{N(x)}^{1/d}$; this implies
\[
\aabs{\iota(\alpha_1) - \iota(\alpha_2)} = \aabs{\iota(x)} \gg_{F} \abs{N(x)}^{1/d} \gg_{F,\frc} p^{M/d}
\]
We may write $\aabs{\iota(\alpha_1) - \iota(\alpha_2)} \geq c_{F,\frc} p^{M/d}$ and set $\delta_{M} := c_{F,\frc} p^{M/d}$ for simplicity.

Around every point $\iota(\alpha)$, we place an open Euclidean ball of radius $\delta_{M}/3$; these balls are pairwise disjoint. By the first paragraph of the proof, all centers $\iota(\alpha)$ lie in the Euclidean ball $\sfB(0,C_{F} X^{1/d})$. So all the smaller balls lie inside the ball $\sfB(0, C_{F}X^{1/d}+\delta_M/3)$. Comparing the $d$-dimensional volumes gives 
\[
A \vol(\sfB(0,\delta_{M}/3)) \leq \vol(\sfB(0,C_{F}X^{1/d}+\delta_{M}/3)),
\]
where $A := \#\{\alpha\in\calR\cap C_F^0:0<\abs{N_{F/\QQ}(\alpha)}\leq X\}$ for simplicity. Since the volume of a $d$-dimensional Euclidean ball of radius $R$ is $\ll_{d} R^{d}$, we have 
\begin{align*}
\#\{\alpha\in\calR\cap C_F^0:0<\abs{N_{F/\QQ}(\alpha)}\leq X\} &\ll_{d} \left(1+ 3 \dfrac{C_{F} X^{1/d}}{\delta_M}\right)^{d} \\ 
&\ll_{F, \frc, d} \left(1 + \left(\dfrac{X}{p^M}\right)^{1/d}\right)^{d} \\
&\ll_{F, \frc, d} 1 + \dfrac{X}{p^{M}},    
\end{align*}
where in the last equality, we use that $(1+t)^d\ll_d1+t^d$ for $t \geq 0$. The proposition is thus proved.
\end{proof}

We also record the following result appearing as \cite[Lemma 3.7]{KwonCorrected}.

\begin{lemma}\label{lem:polynomial-gap}
Fix an integral ideal $\frb$, an element $\gamma\in F$, and an integer $m\geq1$. If $\alpha\in\calO_F\cap C_F^0$ and $0\ne\alpha^m-\gamma\in\frb$, then
\begin{equation}\label{eq:polynomial-gap}
 \abs{N_{F/\QQ}(\alpha)}\gg_{F,\gamma,m}(\Norm\frb)^{1/m}.
\end{equation}
\end{lemma}

\begin{proof}
The ideal divisibility gives
\[
 \Norm\frb\leq\abs{N_{F/\QQ}(\alpha^m-\gamma)}.
\]
Choose one of the finitely many coherent cones $\calB$ containing $\alpha$. Equation \eqref{eq:coherent-coordinate} bounds every $\abs{\tau(\alpha)}$ below by a positive constant depending only on $\calB$. Therefore
\[
 \abs{N(\alpha^m-\gamma)} \leq\abs{N(\alpha)}^m
 \prod_{\tau\in I_F}\left(1+\frac{\abs{\tau(\gamma)}}{\abs{\tau(\alpha)}^m}\right) \ll_{F,\gamma,m}\abs{N(\alpha)}^m.
\]
Combining the two inequalities proves the lemma.
\end{proof}

We then have the following proposition as an analogue of \cite[Proposition 3.8]{KwonCorrected} in the case where $p$ may divide $h_{F}$.

\begin{proposition}\label{prop:norm-gap}
If $\fra$ is an integral ideal of $F$ such that $\fra\ne\calO_F$, $(\fra,\frp)=1$, and $\phi_{\mathrm{av}}(\fra)\ne0$, then
\begin{equation}\label{eq:norm-gap}
 \Norm\fra \gg_{F, p, f} q^{1/(e(p-1))}.
\end{equation}
\end{proposition}

\begin{proof}
Choose $\beta\in\calO_F$ with $\fra^{e}=(\beta)$. Proposition~\ref{prop:power-support} gives $\beta^{p-1}\equiv1\pmod{\frp^{n-c}}$. Multiply $\beta$ by a global unit $u \in \calO_{F}^{\times}$ so that $u\beta \in C_{F}^{0}$. After replacing $\beta$ by $u\beta$, one checks directly that
\[
\beta \in \calO_{F} \cap C_{F}^{0}, \quad \fra^{e} = (\beta), \quad \beta^{p-1} \equiv 1 \pmod{\frp^{n-c}}.
\]

If $\beta^{p-1}=1$, then $(\beta)=\calO_F$, hence $\fra^{e}=\calO_F$. The group of fractional ideals is torsion-free, so $\fra=\calO_F$, contrary to the hypothesis. Thus
\[
 0\ne\beta^{p-1}-1\in\frp^{n-c}.
\]
Apply Lemma~\ref{lem:polynomial-gap} with $\alpha = \beta$, $m=p-1$, $\gamma=1$ and $\frb = \frp^{n-c}$. Since $\abs{N(\beta)}=(\Norm\fra)^{e}$, it gives
\[
 (\Norm\fra)^{e} \gg_{F,p} p^{(n-c)/(p-1)} \gg_{F, p, f} q^{1/(p-1)},
\]
which is \eqref{eq:norm-gap}.
\end{proof}

\subsection{Counting supported ideals}
For $X\geq1$, let
\[
 A_\phi(X)=\#\{0\ne\fra\subseteq\calO_F:\fra\ne\calO_F,\ (\fra,\frp)=1,\ \Norm\fra\leq X,\ \phi_{\mathrm{av}}(\fra)\ne0\}.
\]
The goal of this section is to estimate $A_{\phi}(X)$.

\begin{proposition}\label{prop:support-count}
Uniformly for primitive characters $\phi \in \Xi_{\frp}$ of conductor $\frp^n$, we have
\begin{equation}\label{eq:support-count}
 A_\phi(X)\ll_{F,p,f}1+\frac{X}{q}.
\end{equation}
Moreover, $A_\phi(X)=0$ when $X<C_{0} q^{1/(e(p-1))}$ for a constant $C_{0}$ depending only on $F$, $p$ and $f$.
\end{proposition}

\begin{proof}
Write every ideal in the class represented by $\frc_i^{-1}$ as $\fra=(\alpha)\frc_i^{-1}$ with $\alpha\in\frc_i$. As in the proof of Proposition \ref{prop:norm-gap}, we may take $\alpha\in C_F^0$. By Proposition \ref{prop:power-support},
\[
\alpha^{e(p-1)}\equiv b_i^{p-1}\pmod{\frp^{n-c}}
\]
Then Lemma \ref{lem:power-fibres} shows that there are at most $C_{e,p}$ cosets $\calT_{i,1}, \ldots, \calT_{i, t_i}$ (with $t_i \leq C_{e,p}$) of $\frp^{n-c}\frc_{i}$ in $\frc_{i}$ such that 
\[
\phi_{\mathrm{av}}(\fra) \neq 0 \Rightarrow \alpha \in \bigcup_{j=1}^{t_{i}} \calT_{i, j}.
\]
We emphasize that the actual cosets may depend on $i$ and $n$, but their number is bounded by $C_{e,p} = (p-1)p^{v_{p}(e)}$, which is independent of $i$ and $n$.

Also, we have
\[
 \abs{N_{F/\QQ}(\alpha)}=(\Norm\fra)(\Norm\frc_i)\leq X\Norm\frc_i.
\]
Applying Proposition~\ref{prop:cone-progression} to each coset $\calT_{i,j}$ gives 
\[
\#\left\{ \alpha\in \calT_{i,j}\cap C_F^0: 0<\abs{N(\alpha)}\leq X\Norm\frc_i \right\} \ll_{F,\frc_i}
1+\frac{X\Norm\frc_i}{p^{n-c}}.
\]
For a fixed $i$, there are at most $C_{e,p}$ cosets $\calT_{i,j}$. There are $h_F$ ideal classes. Therefore,
\begin{align*}
A_\phi(X)
&\leq \sum_{i=1}^{h_F} \sum_{j=1}^{t_i}
\#\left\{ \alpha\in\calT_{i,j}\cap C_F^0: \abs{N(\alpha)}\leq X\Norm\frc_i \right\}\\
&\ll_{F} \sum_{i=1}^{h_F} t_i \left( 1+\frac{X\Norm\frc_i}{p^{n-c}} \right) \ll_{F} C_{e,p} \sum_{i=1}^{h_F} \left( 1+\frac{X\Norm\frc_i}{p^{n-c}}\right) \ll_{F, p, f} 1 + \dfrac{X}{q}.
\end{align*}
For the last estimate, the appearance of $f$ in the implied constant comes from the constant $c$ in the definition  $K_0\cap\QQ(\mu_{p^\infty}) = \QQ(\mu_{p^c})$ with $K_{0} = \QQ_{f}(\mu_p)$, and every quantity other than $X$ and $n$ is fixed. The vanishing below $C_{0} q^{1/(e(p-1))}$ follows from Proposition \ref{prop:norm-gap}.
\end{proof}

\subsection{Concluding the ordinary part estimation}

We start by the following partial summation formula.
\begin{lemma}
\label{lem:weighted}
Let $\calS$ be a countable set equipped with a function $H:\calS\rightarrow [R,\infty)$ with $R\geq1$, and define
\[
A_{\calS}(X)=\#\{x\in\calS:H(x)\leq X\}.
\]
Suppose that $A_{\calS}(X)\ll 1+\frac{X}{q}$. Let $0<s<1$ and $y>0$. If $\abs{V(x)}\ll_J(1+x)^{-J}$ for some $J>1-s$, then
\[
\sum_{x\in\calS}H(x)^{-s}\abs{V\left(\frac{H(x)}{y}\right)}
\ll R^{-s}+q^{-1}y^{1-s}.
\]
If instead $A_{\calS}(X)\ll X$, then
\[
\sum_{x\in\calS}H(x)^{-s}\abs{V\left(\frac{H(x)}{y}\right)}
\ll y^{1-s}.
\]
\end{lemma}

\begin{proof}
By the assumed decay property of $V$, it suffices to estimate the sum $\sum_{x\in\calS}g(H(x))$ for $g(u):=u^{-s}(1+u/y)^{-J}$. The function $g$ is positive and decreasing. For $T\geq R$, partial summation with $A_{\calS}(X)=\#\{x\in\calS:H(x)\leq X\}$ gives
\[
\sum_{\substack{x\in\calS\\H(x)\leq T}}g(H(x))
=A_{\calS}(T)g(T)+\int_R^T A_{\calS}(u)(-g^\prime(u))\,\dif u.
\]
Since $A_{\calS}(T)\ll 1+T/q$ and $g(T)\ll_y T^{-s-J}$, the condition $J>1-s$ implies that $A_{\calS}(T)g(T)$ converges to $0$ as $T\to\infty$. Hence
\[
\sum_{x\in\calS}g(H(x))
=\int_R^\infty A_{\calS}(u)(-g^\prime(u))\,\dif u.
\]
Using the assumed estimate on $A_{\calS}(X)$ and the fact that $-g^\prime(u)\geq0$, we obtain
\[
\sum_{x\in\calS}g(H(x))
\ll \int_R^\infty(-g^\prime(u))\,\dif u
+\frac{1}{q}\int_R^\infty u(-g^\prime(u))\,\dif u.
\]
The first integral equals $g(R)$, and $g(R)\leq R^{-s}$. For the second integral, integration by parts gives
\[
\int_R^\infty u(-g^\prime(u))\,\dif u
=Rg(R)+\int_R^\infty g(u)\,\dif u,
\]
because $ug(u)$ converges to $0$ as $u\to\infty$. Writing $z=R/y$, we have
\[
Rg(R)=y^{1-s}z^{1-s}(1+z)^{-J}\ll_{s,J}y^{1-s},
\]
since $J>1-s$. Moreover, the change of variables $u=yv$ gives
\[
\int_R^\infty g(u)\,\dif u =y^{1-s}\int_{R/y}^\infty v^{-s}(1+v)^{-J}\,\dif v \ll_{s,J}y^{1-s}.
\]
The last integral converges because $0<s<1$ and $J>1-s$. Combining the estimates above, we have
\[
\sum_{x\in\calS}H(x)^{-s}\abs{V\left(\frac{H(x)}{y}\right)}
\ll R^{-s}+q^{-1}y^{1-s}.
\]

If $A_{\calS}(X)\ll X$, the same partial-summation argument gives
\[
\sum_{x\in\calS}g(H(x)) \ll \int_R^\infty u(-g^\prime(u))\,\dif u \ll y^{1-s}.
\]
This proves the second assertion.
\end{proof}

Recall that we have defined the ordinary part of the averaged $L$-function in \eqref{eq:S1} as 
\[
    S_1(\phi;y) =\frac{1}{\Gamma(k/2,f)}\sum_{\fra} \frac{a_f(\fra)\calG_{\mathrm{av}}(\phi,\fra)}{(\Norm\fra)^{k/2}} V_{1,k/2}\!\left(\frac{\Norm\fra}{y}\right).
\]
We are now ready to estimate this term.

\begin{theorem}\label{thm:first-half}
For $y\geq1$, we have
\begin{equation*}\label{eq:first-half}
 S_1(\phi;y)=\frac{V_{1,k/2}(1/y)}{\Gamma(k/2,f)}
 +O_{F,f,p,K_0,\varepsilon}\!\left(q^{-\frac{1}{e(p-1)}s_0}+q^{-1}y^{1-s_0}\right).
\end{equation*}
\end{theorem}

\begin{proof}
By Lemma~\ref{lem:partial-gauss-identity}, $\calG_{\mathrm{av}}(\phi,\fra)=\phi_{\mathrm{av}}(\fra)$, and hence
\[
 S_1(\phi;y) = \frac{1}{\Gamma(k/2,f)}\sum_{\fra} \frac{a_f(\fra)\phi_{\mathrm{av}}(\fra)}{(\Norm\fra)^{k/2}} V_{1,k/2}\!\left(\frac{\Norm\fra}{y}\right).
\]
The contribution of the unit ideal is $V_{1,k/2}(1/y)/\Gamma(k/2,f)$ because $a_f(\calO_F)=1$, $\phi_{\mathrm{av}}(\calO_F)=1$, and $\Norm\calO_F=1$. Hence
\[
 S_1(\phi;y) = \frac{V_{1,k/2}(1/y)}{\Gamma(k/2,f)} +E_1(\phi;y),
\]
where
\[
E_1(\phi;y)=\frac{1}{\Gamma(k/2,f)}
\sum_{\fra\neq\calO_F}
\frac{a_f(\fra)\phi_{\mathrm{av}}(\fra)}{(\Norm\fra)^{k/2}}
V_{1,k/2}\!\left(\frac{\Norm\fra}{y}\right).
\]

Recall that $s_0=1/2-t$, where $t=\theta+\varepsilon<1/4$. In particular, $0<s_0<1$. By \eqref{eq:central-coefficient-bound},
\[
 \frac{\abs{a_f(\fra)}}{(\Norm\fra)^{k/2}}
 \ll_{F,f,\varepsilon}(\Norm\fra)^{-s_0}.
\]
Moreover, $\abs{\phi_{\mathrm{av}}(\fra)}\leq1$, and terms with $\phi_{\mathrm{av}}(\fra)=0$ make no contribution. Since $\Gamma(k/2,f)$ is fixed, it follows that
\[
\abs{E_1(\phi;y)}\ll_{F,f,\varepsilon}
\sum_{\substack{\fra\neq\calO_F\\\phi_{\mathrm{av}}(\fra)\neq0}}
(\Norm\fra)^{-s_0}\abs{V_{1,k/2}\!\left(\frac{\Norm\fra}{y}\right)}.
\]
We apply Lemma~\ref{lem:weighted} to $\calS_\phi = \{\fra\neq\calO_F:\phi_{\mathrm{av}}(\fra)\neq0\}$ and $ H(\fra)=\Norm\fra$: Proposition~\ref{prop:support-count} gives
\[
 \#\{\fra\in\calS_\phi:\Norm\fra\leq X\}
 \ll_{F,p,f}1+\frac{X}{q},
\]
while Proposition~\ref{prop:norm-gap} gives, for $\fra \in \calS_{\phi}$, $H(\fra)=\Norm\fra\geq C_0q^{1/(e(p-1))}$. Since $H(\fra)\geq1$ as well, set $R:=\max\{1,C_0q^{\frac{1}{e(p-1)}}\}$. Then $H(\fra)\geq R$. Moreover, \eqref{eq:kernel-bounds} gives
\[
 \abs{V_{1,k/2}(x)}\ll_J(1+x)^{-J}
\]
for every $J>0$. Choose $J>1-s_0$. Applying Lemma~\ref{lem:weighted} with this choice of $R$, and using $R^{-s_0}\leq C_0^{-s_0}q^{-\frac{s_0}{e(p-1)}}$ yields
\[
\sum_{\substack{\fra\neq\calO_F\\\phi_{\mathrm{av}}(\fra)\neq0}}
(\Norm\fra)^{-s_0}\abs{V_{1,k/2}\!\left(\frac{\Norm\fra}{y}\right)}
\ll_{F,p,f,\varepsilon}q^{-\frac{1}{e(p-1)} s_0}+q^{-1}y^{1-s_0}.
\]
Consequently,
\[
E_1(\phi;y)\ll_{F,f,p,K_0,\varepsilon}
q^{-\frac{1}{e(p-1)} s_0}+q^{-1}y^{1-s_0},
\]
which proves the theorem.
\end{proof}
\section{Estimate of the Fricke part}
\label{sec:estimate-fricke}

In this section, we estimate the Fricke part of the Galois-averaged $L$-functions.

\subsection{The Kloosterman sum estimate}

For $z\in(\ZZ/q\ZZ)^\times$, we define the Kloosterman sum as 
\begin{equation}\label{eq:restricted-kloosterman}
 \Kl_c(z;q)=\frac{1}{q}\sum_{\substack{x\bmod q\\x\equiv c_\phi\, \bmod p^c}}
 \rme_q(x+x^{-1}z).
\end{equation}

The following estimate appears as \cite[Lemma 3.5]{KwonCorrected}.
\begin{lemma}\label{lem:restricted-kloosterman}
If $p$ is odd and $\lceil n/2\rceil\geq c$, then $\abs{\Kl_c(z;p^n)}\leq p^{2-n/2}$.
\end{lemma}

\begin{proof}
Under the fixed identifications $F_{\frp}=\QQ_p$, $\calO_{F,\frp}=\ZZ_p$ and $\varpi_{\frp}=p$, the sum $\Kl_c(z;p^n)$ is the complex conjugate of the sum in \cite[Lemma~3.5]{KwonCorrected}, with Kwon's $(n_0,c,d)$ replaced respectively by $(c,c_\phi,z)$. The asserted estimate therefore follows directly from that lemma.
\end{proof}

\subsection{Averaged Gauss sum} \label{sec:translated-local-support}

We prove the following proposition as an analogue of \cite[Proposition 3.6]{KwonCorrected}, allowing the general case when $p$ may divide $h_{F}$.

\begin{proposition}\label{prop:root-average}
If $n>2c$, then for every integral ideal $\fra$,
\begin{equation}\label{eq:root-average}
 \abs{\calG_{\mathrm{av}}^{\iota}(\phi,\fra)}
 \leq(p-1)p^{c+2-n/2}.
\end{equation}
In particular, $\calG_{\mathrm{av}}^{\iota}(\phi,\fra)\ll_{p,f}q^{-1/2}$.
\end{proposition}

\begin{proof}
For a primitive character $\phi \in \Xi_{\frp}$ of conductor $\frp^n$, we have the usual Gauss sum identity $\calG(\phi)\calG(\bar\phi)=q\phi_{\frp}(-1)=q$. Combining it with \eqref{eq:corrected-root-factor} gives, for every $\sigma \in G_{\phi}$, 
\begin{equation*}
\calG(\phi^\sigma)W(\phi^\sigma) = \calG(\bar\phi^{\sigma}) \phi^\sigma(\varpi_{\frp^{2n}\frN})
\end{equation*}
and hence by \eqref{eq:gauss-bridge}, 
\begin{equation}\label{eq:root-simplification}
\calG_{0}(\phi^\sigma)W(\phi^\sigma) = \phi^{\sigma}(\frd_{F})^{2} \calG_{0}(\bar\phi^{\sigma}) \phi^\sigma(\varpi_{\frp^{2n}\frN}).
\end{equation}
The expansions of $\calP_\sigma(\phi)$ and $\calG_{0}(\bar\phi^\sigma)$ in \eqref{eq:Giota-definition}, together with \eqref{eq:root-simplification}, give
\begin{align}\label{eq:Giota-expansion}
\calG_{\mathrm{av}}^{\iota}(\phi,\fra)
&=\frac{1}{q}\sum_{a_1,a}
\overline\phi_{\mathrm{av}}(a_1a\bfd_{\fra,n}\frd_{F}^{-2})\rme_q(a_1+a) \notag \\
&=\sum_{z\in(\ZZ/q\ZZ)^\times}
\overline\phi_{\mathrm{av}}(z\bfd_{\fra,n}\frd_{F}^{-2})\Kl_c(z;q).
\end{align}
Here $a_1 \equiv c_\phi \pmod{p^c}$ and $a \in (\ZZ/q\ZZ)^{\times}$, and $\bfd_{\fra,n}$ is defined as the idele class represented by $\fra\varpi_{\frp^{2n}\frN}^{-1}$. The second equality uses the substitution $z=a_1a$.

Fix $\fra$ prime to $\frp$, and abbreviate $\xi=\overline\phi(\bfd_{\fra,n}\frd_{F}^{-2})$. By Lemma~\ref{lem:trace}, the summand indexed by $z$ in \eqref{eq:Giota-expansion} can be nonzero only if $\xi \overline\phi_{\frp}(z)\in\mu_{p^c}$, or equivalently, 
\begin{equation} \label{eq:translated-local-support}
    \overline\phi_{\frp}(z)\in \xi^{-1}\mu_{p^c}.
\end{equation}
Lemma~\ref{lem:faithful} says that $\overline\phi_{\frp}:(\ZZ/q\ZZ)^\times\to\mu_{p^{n-1}}$ is surjective and has kernel of order $p-1$. The intersection of the coset $\xi^{-1}\mu_{p^c}$ with this image has at most $p^c$ elements. Consequently at most $(p-1)p^c$ values of $z\bmod q$ satisfy \eqref{eq:translated-local-support}. For each such $z$ satisfying \eqref{eq:translated-local-support}, Lemma \ref{lem:trace} gives $\bar\phi_{\mathrm{av}}(z\bfd_{\fra,n}\frd_{F}^{-2})=\bar\phi(z\bfd_{\fra,n}\frd_{F}^{-2})$, and hence $\abs{\bar\phi_{\mathrm{av}}(z\bfd_{\fra,n}\frd_{F}^{-2})} = 1$. Lemma~\ref{lem:restricted-kloosterman} gives $\abs{\Kl_c(z;q)}\leq p^{2-n/2}$. Summing the nonzero summands in \eqref{eq:Giota-expansion} proves \eqref{eq:root-average}.
\end{proof}

\subsection{Concluding the Fricke part estimation}
\label{sec:moment}

Let $\phi\in\Xi_{\frp}$ have conductor $\frp^n$, and put $q=p^n$.

\begin{theorem}\label{thm:dual-half}
If $q^2/y\geq1$ and $n > 2c$, then we have
\begin{equation*}\label{eq:dual-half}
 S_2(\phi;y)\ll_{F,f,p,\varepsilon}
 q^{-1/2}\left(\frac{q^2}{y}\right)^{1-s_0}.
\end{equation*}
\end{theorem}

\begin{proof}
    The factor $\omega_\chi(\varpi_{\frp^n})W_f$ has absolute value $1$, while $\Gamma(k/2,f)$ is fixed. Hence
\[
\abs{S_2(\phi;y)} \ll_f \sum_{\fra} \frac{\abs{a_{W_{\frN}f}(\fra)}}{(\Norm\fra)^{k/2}} \abs{\calG_{\mathrm{av}}^\iota(\phi,\fra)} \abs{ V_{2,k/2}\!\left( \frac{\Norm\fra\,y}{\Norm\frN\,q^2} \right)}.
\]
By \eqref{eq:central-coefficient-bound2} and \eqref{eq:root-average}, with the fixed factor $(p-1)p^{c+2}$ absorbed into the implied constant, we have
\[
\begin{aligned}
\abs{S_2(\phi;y)}
&\ll_{F,f,p,\varepsilon}q^{-1/2}\sum_{\fra}(\Norm\fra)^{-s_0}
\abs{V_{2,k/2}\!\left(\frac{\Norm\fra\,y}{\Norm\frN\,q^2}\right)}\\
&=q^{-1/2}\sum_{\fra}(\Norm\fra)^{-s_0}
\abs{V_{2,k/2}\!\left(\frac{\Norm\fra}{Y}\right)}.
\end{aligned}
\]
Here $Y := \Norm\frN q^{2}/y$. The hypothesis that $q^{2}/y \geq 1$ ensures that $Y \geq 1$ since $\Norm\frN \geq 1$.

Recall the standard ideal-counting estimate (see \cite[Chapter VI, Section 3]{LangANT}):
\begin{equation}\label{eq:all-ideals}
 \#\{0\ne\fra\subseteq\calO_F:\Norm\fra\leq X\}\ll_F X
\end{equation}
Now we use the second part of Lemma~\ref{lem:weighted} with
\[
\calS = \{ 0 \neq \fra \subseteq \calO_{F} \}, \quad H(\fra) = \Norm\fra,
\]
for which the counting function is exactly $A_{\calS}(X) \ll_{F} X$. We take $s = s_0, V = V_{2, k/2}$ and $y = Y$. Since $0 < s_0 < 1$ and $V_{2, k/2}$ satisfies the required decay property \eqref{eq:kernel-bounds}, the estimate \eqref{eq:all-ideals}, together with an application of Lemma \ref{lem:weighted} to these parameters, gives
\[
\sum_{\fra} (\Norm\fra)^{-s_0} \abs{ V_{2,k/2}\!\left(\dfrac{\Norm\fra}{Y}\right) } \ll_{F,f,\varepsilon} Y^{1-s_0}.
\]
Multiplying by $q^{-1/2}$ then proves the theorem.
\end{proof}
\section{Main theorems}
\label{sec:comparison}

\subsection{The main estimate and the eventual nonvanishing theorem}

Combining Theorems~\ref{thm:first-half} and~\ref{thm:dual-half} gives the following main estimate of this article.
\begin{theorem}[{The main estimate (Theorem \ref{main:B})}] \label{thm:average}
Let $f$ be a normalized cuspidal Hecke eigenform on $\GL(2)$ over $F$ of weight $(\bfk, \bfm)$, archimedean type $J$, level $\frN$ and central character $\chi$ in the sense of Definition \ref{def:cusp-form}. Assume that $p$ is an odd prime number and that $\frp \mid p$ is a totally split prime ideal of $F$, with $\frp$ coprime to $\frN$. 

Let $\phi\in\Xi_{\frp}$ have conductor $\frp^n$, where $n>2c$ with the constant $c$ defined in \eqref{eq:cyclotomic-intersection-intro}. For $\varepsilon>0$ with $t=\frac{7}{64}+\varepsilon<1/4$, put $s_0=1/2-t$ and let $q=p^n$. For $q<y<q^2$, we have the following estimate:
\begin{equation}\label{eq:main-average}
 L_{\mathrm{av}}(f,\phi)
 =\frac{V_{1,k/2}(1/y)}{\Gamma(k/2,f)}
 +O_{F,f,p,\varepsilon}\!\left(q^{-\frac{1}{e(p-1)}s_0}+q^{-1}y^{1-s_0}
 +q^{-1/2}\left(\frac{q^2}{y}\right)^{1-s_0}\right).
\end{equation}
In particular, as $n \rightarrow \infty$, $L_{\mathrm{av}}(f,\phi)$ converges to $1$ uniformly for $\phi\in\Xi_{\frp}$. More precisely, for $n \to \infty$, we have 
\[
\sup_{\phi\in\Xi_{\frp}, \, \cond(\phi)=\frp^n}
\abs{L_{\mathrm{av}}(f,\phi)-1}
\longrightarrow 0
\]
\end{theorem}

\begin{proof}
Combining Theorems~\ref{thm:first-half}
and~\ref{thm:dual-half} gives the estimate
\eqref{eq:main-average}; the implied constant is independent of $y$, $n$, and the particular character $\phi$.

It remains to choose the auxiliary parameter $y$ in order to deduce the asserted convergence. Fix a real number $a$ and put $y=q^a$. Since $1-s_0=\frac{1}{2}+t$, the exponents of $q$ in the two $y$-dependent error terms are $a\left(\frac{1}{2}+t\right)-1$ and $-\frac{1}{2}+(2-a)\left(\frac{1}{2}+t\right)$. One checks that both of them are negative provided
\[
a \in \left( \frac{1/2+2t}{1/2+t}, \frac{1}{1/2+t} \right).
\]
This interval is nonempty because $t<1/4$ in our convention \eqref{eq:range-t}. Moreover, since $t>0$, the interval is contained in $(1,2)$, so the corresponding choice $y=q^a$ satisfies $q<y<q^2$.

For any fixed $a$ in this interval, all three error terms in \eqref{eq:main-average} tend to zero as $q\to\infty$. Furthermore, \eqref{eq:kernel-bounds} gives (by taking $x$ there as $1/y$, and further $y = q^{a}$):
\[
\frac{V_{1,k/2}(1/y)}{\Gamma(k/2,f)} = 1+O(y^{-1/2}) = 1+O(q^{-a/2}).
\]
All implied constants are independent of $\phi$. The convergence result then follows.
\end{proof}

\begin{theorem}[{Eventual nonvanishing theorem (Theorem \ref{main:A})}]\label{thm:nonvanishing}
Under the hypotheses of Theorem \ref{thm:average},
\[
 L\!\left(\frac{k}{2},f\otimes\phi\right)\ne0
\]
for every $\phi\in\Xi_{\frp}$ of sufficiently large conductor.
\end{theorem}

\begin{proof}
It follows from Hida's algebraicity theorem \cite[Theorem~8.1]{Hida} that 
\begin{equation} \label{eq:hida-original}
 L(k/2,f\otimes\phi)=0 \quad\Longleftrightarrow\quad L(k/2,f^\sigma\otimes\phi^\sigma)=0
\end{equation}
for any $\sigma \in \Aut_{\QQ}(\CC)$. Put $K_0=\QQ_f(\mu_p)$ and $K=\QQ(\phi)$. By Lemma \ref{lem:faithful}, the image of the local restriction of $\phi$ contains $\mu_{p^{n-1}}$. Since $n>2c$, we have $\QQ(\mu_{p^c})\subseteq K$, and hence
\[
K_0\cap K =\bigl(K_0\cap\QQ(\mu_{p^\infty})\bigr)\cap K =\QQ(\mu_{p^c}).
\]
Every $\sigma\in G_\phi=\Gal(K/(K_0\cap K))$ therefore extends to an automorphism $\widetilde\sigma$ of $\CC$ fixing $K_0$. In particular, $f^{\widetilde\sigma}=f$. Then \eqref{eq:hida-original} now gives, for $\sigma \in G_{\phi}$,
\[
L(k/2,f\otimes\phi)=0 \quad\Longleftrightarrow\quad L(k/2,f\otimes\phi^\sigma)=0.
\]

If $L(k/2, f \otimes \phi) = 0$ for infinitely many Hecke characters $\phi \in \Xi_{\frp}$, then it follows that $L_{\mathrm{av}}(f,\phi)=0$ for infinitely many such $\phi$, contradicting the conclusion of Theorem \ref{thm:average} that $L_{\mathrm{av}}(f, \phi) \neq 0$ for sufficiently large $n$.
\end{proof}

\subsection{Application to $\frp$-ramified $\ZZ_p$-extensions of imaginary quadratic fields}
\label{sec:qi-stokes}

In this section, we apply the eventual nonvanishing result (Theorem \ref{thm:nonvanishing}) to the $\frp$-ramified $\ZZ_p$-extensions of imaginary quadratic fields.

\subsubsection{The $\frp$-ramified $\ZZ_p$-extension}

Let $K$ be an imaginary quadratic field, and suppose that the odd prime $p$ splits as $p\calO_K=\frp\bar\frp$. Taking $S=\{\frp\}$, the theorem of Qi and Stokes \cite[Theorems~1.1 and~2.1]{QiStokes} gives a unique $S$-ramified $\ZZ_p$-extension $K_{\infty}^{(\frp)}$ of $K$. \footnote{Indeed, their construction applies to CM fields under Leopoldt's conjecture, which is automatic for the imaginary quadratic field $K$}.

For later use, we describe this $\ZZ_p$-extension more explicitly. Let $M_{\{\frp\}}$ be the maximal abelian pro-$p$ extension of $K$ unramified outside $\frp$, and put $H_{\frp}=\Gal(M_{\{\frp\}}/K)$. Class field theory puts $H_{\frp}$ into the following short exact sequence \footnote{We add more details here. For $m\geq1$, recall that we denote by $\Cl(K, \frp^m)$ the ray class group of $K$ with modulus $\frp^m$. Global class field theory gives $H_{\frp}\cong \varprojlim_m\Cl(K, \frp^m)[p^\infty]$. The $p$-primary part of the ray class sequence is
\[
1\rightarrow (1+p\ZZ_p)/(1+p^m\ZZ_p) \rightarrow \Cl(K, \frp^m)[p^\infty] \rightarrow \Cl(K)[p^\infty] \rightarrow 1.
\]
Here we have used $K_{\frp}\cong\QQ_p$ and the fact that $\calO_K^\times$ has order prime to $p$.
Passing to the inverse limit over $m$ gives the exact sequence \eqref{eq:pro-p-imaginary}.}:
\begin{equation}\label{eq:pro-p-imaginary}
 1\rightarrow 1+p\ZZ_p \rightarrow H_{\frp}\rightarrow \Cl(K)[p^\infty] \rightarrow 1.
\end{equation}
In particular, $H_{\frp}$ is a finitely generated abelian pro-$p$ group of $\ZZ_p$-rank one. If $T_{\frp}$ denotes the finite torsion subgroup of $H_{\frp}$, then $\Gamma_{\frp}:=H_{\frp}/T_{\frp} \cong \ZZ_p$. The fixed field of $T_{\frp}$ is therefore a $\ZZ_p$-extension of $K$ unramified outside $\frp$. By \cite[Theorems 1.1 and 2.1]{QiStokes}, it is their $\frp$-ramified $\ZZ_p$-extension $K_{\infty}^{(\frp)}/K$.

\subsubsection{Characters over the $\frp$-ramified $\ZZ_p$-extension and their conductor} \label{sec:conductor-frp-Zp}
Consider the composite homomorphism $\iota: 1+p\ZZ_p\rightarrow H_{\frp}\rightarrow\Gamma_{\frp}$. Its kernel is the intersection of the image of $1+p\ZZ_p$ in $H_{\frp}$ with $T_{\frp}$. Since $(1+p\ZZ_p) \cap T_{\frp} = \{1\}$, this implies that the composite $\iota$ is injective. Its image is open because
\[
\Gamma_{\frp}/\iota(1+p\ZZ_p) \cong H_{\frp}/\bigl(T_{\frp}(1+p\ZZ_p)\bigr)
\]
is a quotient of the finite group
$H_{\frp}/(1+p\ZZ_p)\cong\Cl(K)[p^\infty]$ by \eqref{eq:pro-p-imaginary}. The $p$-adic logarithm gives an isomorphism $\lambda:1+p\ZZ_p\xrightarrow{\sim}\ZZ_p$ given by $\lambda(x)=\log_p(x)/p$. It satisfies $\lambda(1+p^m\ZZ_p)=p^{m-1}\ZZ_p$. Choose also an isomorphism $\jmath:\Gamma_{\frp}\xrightarrow{\sim}\ZZ_p$. Then the continuous open injection $\jmath\circ\iota\circ\lambda^{-1}:\ZZ_p\rightarrow\ZZ_p$ has the form
\[
(\jmath\circ\iota\circ\lambda^{-1})(x)=p^bu x
\]
for some $b\geq0$ and $u\in\ZZ_p^\times$.

Now let $\phi$ be a continuous character of $\Gamma_{\frp}$ of exact order $p^r$.
Under the identification $\jmath$, the kernel of $\phi$ is $p^r\ZZ_p$.
If $r>b$, then the kernel of its pullback through $\jmath\circ\iota\circ\lambda^{-1}$ is
\[
\{x\in\ZZ_p:p^bu x\in p^r\ZZ_p\} =p^{r-b}\ZZ_p.
\]
Therefore, the character $\phi$ is trivial on $1+p^{m}\ZZ_p$ precisely when $m-1 \geq r-b$, i.e. it is nontrivial on $1+p^{r-b}\ZZ_p$ and trivial on $1+p^{r-b+1}\ZZ_p$. Hence its exact local conductor exponent is $r-b+1$ when $r > b$. Globally, it is unramified outside $\frp$; therefore, the conductor of $\phi$ is $\frp^{r-b+1}$.

In particular, the characters of $\Gamma_{\frp}$ of exact order $p^r$, for $r>b$, give an infinite family of $p$-power-order Hecke characters in $\Xi_{\frp}$, and their conductors tend to infinity when $r$ goes to infinity.

\begin{remark}
By construction, the prime ideal $\bar\frp$ is unramified in the $\ZZ_p$-extension $K_{\infty}^{(\frp)}/K$. It is therefore different from the cyclotomic and anticyclotomic $\ZZ_p$-extensions.
\end{remark}

\subsubsection{Base change of elliptic curves} \label{sec:base-change}
Let $E/\QQ$ be an elliptic curve of conductor $N$. By the modularity theorem \cite{BCDT}, there is a normalized newform $f_E\in S_2(\Gamma_0(N))$ such that $L(E/\QQ,s)=L(f_E,s)$. Let $\pi_E=\pi_{f_E}$ be the associated cuspidal automorphic representation of $\GL_2(\AA_\QQ)$. 

The work of Langlands \cite{Langlands} gives the quadratic base change $\Pi := \BC_K(\pi_E)$ of $\pi_{E}$ as an automorphic representation of $\GL_2(\AA_K)$, and the representation $\Pi$ is cuspidal unless $E$ has CM by $K$. Let $f_{E,K}$ denote the normalized newform generating $\Pi$. If $p\nmid N$ and $\frp\mid p$ is split in $K$, then $K_{\frp}\simeq\QQ_p$ and $\Pi_{\frp}$ is unramified. 

The $L$-function of $f_{E,K}$ has central point $s=1$. For every finite-order Hecke character $\phi$ of $K$ one has
\begin{equation} \label{eq:L-function-identity}
 L(E/K,\phi,s)=L(s,f_{E,K}\otimes\phi).
\end{equation}

\subsubsection{The main result}

The following theorem is then a corollary of Theorem \ref{thm:nonvanishing}. 

\begin{theorem} \label{thm:elliptic}
Let $E/\QQ$ be an elliptic curve of conductor $N$. Suppose $p \nmid N$, $p$ splits in $K$ and $E$ does not have CM by $K$. For every finite-order character $\phi$ of $\Gamma_{\frp}$ of sufficiently large conductor, $L(E/K,\phi,1)\ne0$.
\end{theorem}

\begin{proof}
From Section \ref{sec:conductor-frp-Zp}, the characters of $\Gamma_{\frp}$ are $p$-power characters of $\frp$-power conductor and their conductors are unbounded. We have seen in Section \ref{sec:base-change} that, under the theorem's assumptions, $p$ is unramified in $K$, $\Pi$ is cuspidal and $\Pi_{\frp}$ is unramified. With equation \eqref{eq:L-function-identity}, Theorem~\ref{thm:nonvanishing} applies directly.
\end{proof}

In a forthcoming work with Peikai Qi, we establish the Iwasawa theory of elliptic curves $E/\QQ$ with good ordinary reduction at $p$ along the $\frp$-ramified $\ZZ_p$-extension $K_{\infty}^{(\frp)}/K$ satisfying additional hypotheses covered by Theorem \ref{thm:elliptic}. The eventual nonvanishing property of $L(E/K, \phi, 1)$ in Theorem \ref{thm:elliptic} serves as a crucial analytic input in such an Iwasawa theory. As an application, we shall prove that under certain technical assumptions covered by Theorem \ref{thm:elliptic}, the Mordell--Weil group $E(K_{\infty}^{(\frp)})$ is a finitely generated abelian group.

\printbibliography

\end{document}